\documentclass{gOMS2e}
\usepackage{epstopdf}
\usepackage{subfigure}
\theoremstyle{plain}
\newtheorem{theorem}{Theorem}[section]

\newtheorem{lemma}[theorem]{Lemma}

\theoremstyle{definition}
\newtheorem{definition}{Definition}
\theoremstyle{remark}

\usepackage{cmap}					
\usepackage{mathtext} 				
\usepackage{amsmath,amsfonts,amssymb,amsthm,mathtools}
\usepackage[noend]{algpseudocode}
\usepackage[usenames]{color}
\usepackage[ruled]{algorithm2e}
\usepackage[hidelinks]{hyperref}
\usepackage{multirow}
\renewcommand{\leq}{\leqslant}
\renewcommand{\geq}{\geqslant}

\newcommand{\argmin}{\operatornamewithlimits{argmin}}

\renewcommand {\*} {\ast}

\newcommand{\sm}[2]{\begin{smallmatrix}\item1\\\item2 \end{smallmatrix}}

\providecommand{\floor}[1]{\left \lfloor \item1 \right \rfloor }
\providecommand{\ceil}[1]{\left \lceil \item1 \right \rceil }

\newcommand\ddfrac[2]{\frac{\displaystyle \item1}{\displaystyle \item2}}
\newcommand\addtag{\refstepcounter{equation}\tag{\theequation}}
\def\eps{\varepsilon}
\newcommand{\mc}[1]{\mathbb{\item1}}
\newcommand{\circled}[1]{\raisebox{.5pt}{\textcircled{\raisebox{-.9pt} {#1}}}}

\begin{document}

\let\oldproofname=\proofname
\renewcommand{\proofname}{\rm\bf{\oldproofname}}
\title{A universal modification of the linear coupling method}
\author{
\name{Sergey Guminov\textsuperscript{a}$^{\ast}$\thanks{Corresponding author. Email: sergey.guminov@phystech.edu}, Alexander Gasnikov\textsuperscript{a,b}, Anton Anikin\textsuperscript{c},
Alexander Gornov\textsuperscript{c}}
\affil{\textsuperscript{a} Moscow Institute of Physics and Technology, Moscow, Russia\\\textsuperscript{b} Institute for Information Transmission Problems RAS, Moscow, Russia\\\textsuperscript{c}Matrosov Institute for System Dynamics and Control Theory of Siberian Branch of Russian Academy of Sciences, Irkutsk}}

\maketitle
\begin{abstract}
In the late sixties, N. Shor and B. Polyak independently proposed optimal first-order methods for solving non-smooth convex optimization problems. In 1982 A. Nemirovski proposed optimal first-order methods for solving smooth convex optimization problems, which utilized auxiliary line search. In 1985 A. Nemirovski and Yu. Nesterov proposed a parametric family of optimal first-order methods for solving convex optimization problems with intermediate smoothness. In 2013 Yu. Nesterov proposed a universal gradient method which combined all good properties of the previous methods, except the possibility of using auxiliary line search. One can typically observe that in practice auxiliary line search improves performance for many tasks. In this paper, we propose the apparently first such method of non-smooth convex optimization allowing the use of the line search procedure. Moreover, it is based on the universal gradient method, which does not require any a priori information about the actual degree of smoothness of the problem. Numerical experiments demonstrate that the proposed method is, in some cases, considerably faster than Nesterov's universal gradient method.
\end{abstract}
\begin{keywords}
Convex optimization; First-order methods; Non-smooth optimization; Line search
\end{keywords}
\begin{classcode}
90C25, 68Q25
\end{classcode}

\section{Introduction}
Traditionally, convex optimization problems have been divided into two main classes: the class of smooth problems and the class of non-smooth problems \cite{polyakintroduction}. However, introducing an intermediate class of problems with convex differentiable objectives with H$\ddot{\text{o}}$lder continuous gradient allows us to view the classes of smooth and non-smooth convex optimization problems as two extreme cases of this intermediate class. 

The first optimal methods for this class were introduced in \cite{nemirovski1985optimal}. However, both these procedures and some others presented later had a serious drawback: they required too much information about the objective (for example, the degree of the objective function's smoothness or the distance from the initial point to the solution) to be used efficiently. 

In \cite{nesterov2015universal} the Universal Fast Gradient Method is presented. It is optimal for the class of problems with convex differentiable objectives with H$\ddot{\text{o}}$lder continuous gradient, has a low iteration cost, and does not involve any parameters dependent on the objective.

Some minimization methods allow for the use of an exact line search procedure. A classic example of such a method is the steepest descent method, which is a version of the gradient descent method in which on each iteration instead of performing a step of fixed length in the direction of the negative gradient the objective function is minimized along said direction. Although this does not improve the worst-case convergence rate, such line search procedures often perform very well in practice. The aim of this work was to construct a universal method which allowed for the use of an exact line search procedure. By combining the core idea of Nesterov's Universal Fast Gradient Method with the framework described by Allen-Zhu et al. in \cite{allen2014linear}, such a method was devised. As far as it is known to the authors of this paper, our work contains the first example of such a method, although a method utilising exact line search for solving minimization problems with convex Lipschitz continuous objectives was recently constructed by Drori et al \cite{drori2018efficient}. Their work also contains an example of a universal method which uses an exact three dimension subspace minimization on each iteration. Our numerical experiments indicate that the exact line search step does indeed demonstrate great performance on some non-smooth problems. Note that in the well-known Shor's type methods with variable metric for non-smooth convex optimization problems line search is performed not in the direction of the negative gradient. These methods also require quadratic memory  \cite{polyakintroduction}.  

The paper is organized as follows. Firstly, we define the intermediate class of problems which we refer to above, set the problem and give other definition used later in this paper. Secondly, we define Nesterov's Universal Fast Gradient Method, which we will be using as a benchmark in our numerical experiments. In \textbf{Section 2} we present our Universal Linear Coupling Method, prove its convergence and equip it with a stopping criterion. \textbf{Section 3} contains notes on how to implement the line search procedure and how its accuracy affects the method's convergence. Finally, \textbf{Section 4} is dedicated to the results of our numerical experiments.

\subsection{Preliminaries}

One of the conditions often used in convergence analysis of numerical optimization methods is $L$-smoothness.

\begin{definition}
A  function $f:\ \mathbb{R}^n\rightarrow \mathbb{R}^m$ is called Lipschitz continuous with constant $L$ if

\[\|f(x)-f(y)\|\leq L\|x-y\|\quad \forall x,y\in\mathbb{R}^n.\]

\end{definition}

\begin{definition}
A differentiable function $f:\ \mathbb{R}^n\rightarrow \mathbb{R}^m$ is called $L$-smooth if its gradient is Lipschitz continuous with constant $L$:
\[\|\nabla f(x)-\nabla f(y)\|\leq L\|x-y\|\quad \forall x,y\in\mathbb{R}^n.\]
\end{definition}

We will be using the following natural generalisation of Lipschitz continuity.

\begin{definition}
A  function $f:\ \mathbb{R}^n\rightarrow \mathbb{R}^m$ satisfies the H$\ddot{\text{o}}$lder condition (or is H$\ddot{\text{o}}$lder continuous) if there exist constants $\nu\in[0,1]$ and $M_\nu\geqslant 0$, such that \[\|f(x)-f(y)\| \leqslant M_\nu\|x-y\|^\nu\quad \forall\ x,y\in\mathbb{R}^n.\]
\end{definition} 

The constant $\nu$ in this definition is called the exponent of the H$\ddot{\text{o}}$lder condition. H$\ddot{\text{o}}$lder continuity coincides with Lipschitz continuity if $\nu=1$. On the other hand, H$\ddot{\text{o}}$lder continuity with $\nu=0$ is just boundedness. If a function is differentiable and its gradient is H$\ddot{\text{o}}$lder continuous, then exponent $\nu$ is a measure of the function's smoothness.

Throughout this paper we will be working with the problem \[f(x)\rightarrow \min_{x\in\mathbb{R}^n}, \] where $f(x)$ is a convex differentiable function and its gradient satisfies the H$\ddot{\text{o}}$lder condition for some $\nu\in[0,1]$ with some constant $M_{\nu}$. We denote some solution to this problem as $x^\ast$.

Let us define Bregman divergence $V_x(y)$ as follows:

\[V_x(y)=\omega(y)-\langle\nabla\omega(x),y-x\rangle-\omega(x), \] where $\omega(x)$ is a 1-strongly convex function. $\omega$ is also called a distance generating function. By definition, 

\[V_x(y)\geqslant \frac{1}{2}\|y-x\|^2.\]

\subsection{Universal Method}

In \cite{Devolder2014} it is shown that the notion of inexact oracle allows one to apply some methods of smooth convex optimization to non-smooth problems. The following lemma plays a key role in this:

\begin{lemma}
Let function $f$ be differentiable and have H$\ddot{o}$lder continuous gradient. Then for any $\delta>0$ we have \[f(y)\leqslant f(x)+\langle\nabla f(x), y-x\rangle+\frac{M}{2}\|y-x\|^2+\frac{\delta}{2},\] where \[M=M\left(\delta,\nu, M_\nu\right)=\left[\frac{1-\nu}{1+\nu}\frac{M_\nu}{\delta}\right]^{\frac{1-\nu}{1+\nu}}M_\nu.\]

\end{lemma}
The exact values $\left(f(x),\nabla f(x)\right)$ of a differentiable function $f$ with H$\ddot{\text{o}}$lder continuous gradient allow us to obtain an upper bound similar to the one obtained by using inexact information for a differentiable and L-smooth function. This allows one to apply methods reliant on the usage of an inexact oracle for L-smooth objectives to optimize objectives with H$\ddot{\text{o}}$lder continuos gradient.

However, knowledge of the parameters $\nu$ and $M_\nu$ from the definition of H$\ddot{\text{o}}$lder continuity is still required to apply such an approach. In \cite{nesterov2015universal} a line search procedure was used to estimate the needed parameters similarly to how the constant of $L$-smoothness is estimated in adaptive methods. For a general norm on $\mathbb{R}^n$ and a corresponding Bregman divergence $V_x(y)$ the Universal Fast Gradient Method may be written as follows.
\newpage
\begin{algorithm}
    \SetKwInOut{Input}{Input}
    \SetKwInOut{Output}{Output}
	
    \caption{UFGM($f$, $L_0$, $x_0$, $\eps$, $T$)}
    \Input{$f$ a differentiable convex function with H$\ddot{\text{o}}$lder continuous gradient;
    initial value of the "inexact" Lipschitz continuity constant $L_0$;
    initial point $x_0$;
    accuracy $\eps$;
    number of iterations $T$.}
    $y_0\gets x_0$, $z_0\gets x_0$, $\alpha_0 \gets 0$, $\psi_0(x)\gets V_{x_0}(x)$\\
  \For{$k=0$ to $T-1$}{
  	$L_{k+1}\gets\frac{L_{k}}{2}$\\
    \While{True}{
    $v_k=\argmin\limits_{x\in \mathbb{R}^n} \psi_k(x)$\\
    $\alpha_{k+1}\gets\frac{1}{2L_{k+1}}+\sqrt{\frac{1}{4L^2_{k+1}}+\alpha^2_k\frac{L_k}{L_{k+1}}}$\\
    $\tau_k\gets\frac{1}{\alpha_{k+1}L_{k+1}}$\\
    $x_{k+1}\gets\tau_kv_k+(1-\tau_k)y_k$\\
    $z_{k+1}\gets \argmin\limits_{z\in \mathbb{R}^n} \alpha_{k+1}\langle\nabla f(x_{k+1}), z-v_k\rangle +V_{v_k}(z)$\\
    $y_{k+1}\gets\tau_kz_{k+1}+(1-\tau_k)y_k$\\
    \If{$f(y_{k+1})\leqslant f(x_{k+1})+\langle\nabla f(x_{k+1}),y_{k+1}-x_{k+1}\rangle+\frac{L_{k+1}}{2}\|y_{k+1}-x_{k+1}\|^2+\frac{\tau_k\eps}{2}$}{\textbf{break}}
    \Else{$L_{k+1}\gets 2L_{k+1}$}
  	}
    $\psi_{k+1}(x)\gets\psi_k(x)+\alpha_{k+1}\left[f(x_{k+1})+\langle\nabla f(x_{k+1}),x-x_{k+1}\rangle\right]$  
  }
  \Return{$y_T$}
  
\end{algorithm}

The above method does not require a priori knowledge of the smoothness parameter $\nu$ or the corresponding $M_\nu$. The following theorem gives the convergence rate of the above algorithm:

\begin{theorem}
Let f be a differentiable convex function with H$\ddot{\text{o}}$lder continuous gradient with some exponent $\nu$ and $M_\nu<\infty$. Let $L_0\leqslant M(\eps,\nu,M_\nu)$. Then
\[f(y_k)-f(x^\ast) \leq \left[\frac{2^{2+4\nu}M_\nu^2}{\eps^{1-\nu}k^{1+3\nu}}\right]^{\frac{1}{1+\nu}} +\frac{\eps}{2}.\]

\end{theorem} 

What follows is that one may obtain an $\eps$-accurate solution in \[k\leqslant \inf_{\nu\in[0,1]} \left[\left(\frac{2^\frac{3+5\nu}{2}M_\nu}{\eps}\right)^\frac{2}{1+3\nu} \left(\frac{1}{2}\|x_0 -x^\ast\|^2\right)^{\frac{1+\nu}{1+3\nu}}\right]\] iterations. If the problem admits multiple solutions, then $x^\ast$ may be considered to be the solution minimizing $\frac{1}{2}\| x_0 - x^\ast \|^2$. As shown in \cite{nemirovskii1983problem}, this is optimal up to a multiplicative constant independent of the accuracy, the initial point, and the objective function.

\section{Universal Linear Coupling Method}

We are now ready to present our universal method based on the linear coupling method proposed by Allen-Zhu et al. \cite{allen2014linear} The Linear Coupling framework is chosen as a basis for our method because it allows for the usage of an exact line search step, which is our goal. The original linear coupling method utilizes gradient and mirror descent steps to guarantee optimal convergence rate for convex objectives. However, it is clear from the convergence analysis of said method that the gradient step is only used to obtain a lower bound on the decrease of the objective during this step. This means that any procedure capable of guaranteeing at least such a decrease may be utilized instead. Since in the unconstrained Euclidean setting the gradient step is always performed in the direction of the negative of the gradient, one may use the steepest descent method instead. This idea combined with the idea of Nesterov's universal method allows us to modify the Linear Coupling method in the following way:

\begin{algorithm}
    \SetKwInOut{Input}{Input}
    \SetKwInOut{Output}{Output}
	
    \caption{ULCM($f$, $L_0$, $x_0$, $\eps$, $T$)}
    \Input{$f$ a differentiable convex function with H$\ddot{\text{o}}$lder continuous gradient;
    initial value of the "inexact" Lipschitz continuity constant $L_0$;
    initial point $x_0$;
    accuracy $\eps$;
    number of iterations $T$.}
    $y_0 \gets x_0$, $z_0 \gets x_0$, $\alpha_0 \gets 0$\\
  \For{$k=0$ to $T-1$}{
  	$L_{k+1}\gets\frac{L_{k}}{2}$\\
    \While{True}{
  	$\alpha_{k+1}\gets\frac{1}{2L_{k+1}}+\sqrt{\frac{1}{4L^2_{k+1}}+\alpha^2_k\frac{L_k}{L_{k+1}}}$\\
    $\tau_k\gets\frac{1}{\alpha_{k+1}L_{k+1}}$\\
    $x_{k+1}\gets\tau_kz_k+(1-\tau_k)y_k$\\
    $h_{k+1}\gets\argmin\limits_{h\geqslant 0} f(x_{k+1}-h\nabla f(x_{k+1}))$\\
    $y_{k+1}\gets x_{k+1}-h_{k+1}\nabla f(x_{k+1})$\\
    $z_{k+1}\gets z_k-\alpha_{k+1}\nabla f(x_{k+1})$\\
    \If{$\langle \alpha_{k+1}\nabla f(x_{k+1}),z_k-z_{k+1}\rangle-\frac{1}{2}\|z_k-z_{k+1}\|^2\leq \alpha^2_{k+1}L_{k+1}(f(x_{k+1})-f(y_{k+1})+\frac{\tau_k\eps}{2})$}{\textbf{break}}
    \Else{$L_{k+1}\gets 2L_{k+1}$}
  	}
 
  }
  \Return{$y_T$}
\end{algorithm}
\newpage

As far as it is known to the authors of this paper, this is the first universal method of non-smooth optimization utilizing steepest descent steps.
 
From this point onwards $L_k$ will always denote the value obtained at the end of a full iteration of the "for" loop.

We shall now show that the above algorithm is well-defined. To be more precise, we shall prove that the if-condition inside the while loop is satisfied after a finite number of iterations for any $k$. 

\begin{lemma}
$f(x)$ is a convex differentiable function and its gradient satisfies the H$\ddot{\text{o}}$lder condition for some $\nu\in[0,1]$ with some constant $M_{\nu}$. Then for all steps k of above algorithm

\[\alpha_{k+1}\langle\nabla f(x_{k+1}),z_k-z_{k+1}\rangle-\frac{1}{2}\|z_k-z_{k+1}\|^2\leqslant \alpha^2_{k+1}L_{k+1}\left(f(x_{k+1})-f(y_{k+1})+\frac{\tau_k\eps}{2}\right), \] for all $L_{k+1}$ satisfying

\[L_{k+1}\geqslant M(\tau_k\eps,\nu,M_\nu)=\left[\frac{1-\nu}{1+\nu}\frac{M_\nu}{\tau_k\eps}\right]^{\frac{1-\nu}{1+\nu}}M_\nu.\]

\end{lemma}
\begin{proof}

\begin{align*}
\alpha_{k+1}\langle\nabla f(x_{k+1}),z_k-z_{k+1}\rangle-\frac{1}{2}\|z_k-z_{k+1}\|^2\\\leqslant\frac{\alpha^2_{k+1}}{2}\|\nabla f(x_{k+1})\|^2&\leqslant M\alpha_{k+1}^2\left(f(x_{k+1})-f(y_{k+1})+\frac{\tau_k\eps}{2})\right)
\end{align*}

Here the first inequality follows from the fact that $\|\alpha_{k+1}\nabla f(x_{k+1})-(z_k-z_{k+1})\|^2\geq 0$. To get the last inequality we will use \textsc{Lemma 1.1} with $\delta=\tau_k\eps$ and $x=x_{k+1}$, $y=x_{k+1}-\beta\nabla f(x_{k+1})$: 

\begin{align*}
f(y)&\leqslant f(x_{k+1})+\langle\nabla f(x_{k+1}), -\beta\nabla f(x_{k+1})\rangle+\frac{\beta^2M}{2}\|\nabla f(x_{k+1})\|^2+\frac{\tau_k\eps}{2}\\ &= f(x_{k+1})-\beta\|\nabla f(x_{k+1})\|^2+\frac{\beta^2M}{2}\|\nabla f(x_{k+1})\|^2+\frac{\tau_k\eps}{2}.
\end{align*}
Minimising the right-hand side over $\beta\in\mathbb{R}$, we get $\beta=\frac{1}{M}$. This results in the following guarantee:

\[f(y)-f(x_{k+1})\leqslant -\frac{\|\nabla f(x_{k+1})\|^2}{2M}+\frac{\tau_k\eps}{2}. \]

In our algorithm \begin{align*}
y_{k+1}= x_{k+1}-h_{k+1}\nabla f(x_{k+1}),\\h_{k+1}=\argmin\limits_{h\geqslant 0} f(x_{k+1}-h\nabla f(x_{k+1})),
\end{align*} so \[f(y_{k+1})-f(x_{k+1})\leq f(y)-f(x_{k+1})\leq -\frac{\|\nabla f(x_{k+1})\|^2}{2M}+\frac{\tau_k\eps}{2}.\]

\end{proof}
\subsection{Comparison with the UFGM method}

Note that in the case of Euclidean norm and $V_x(y)=\frac{1}{2}\|x-y\|^2$, in the UFGM algorithm the mirror descent step \[z_{k+1}\gets \argmin\limits_{z\in \mathbb{R}^n} \alpha_{k+1}\langle\nabla f(x_{k+1}), z-v_k\rangle +V_{v_k}(z)\] may be rewritten as \[z_{k+1}\gets v_k-\alpha_{k+1}\nabla f(x_{k+1}).\] Moreover, in the case of the Euclidean norm the sequence $\{v_k\}$ turns out to be identical to the sequence $\{z_k\}$. Now by direct substitution of $z_{k+1}$ and by using $(1-\tau_k)y_k=x_{k+1}-\tau_kv_k$ we get that \[y_{k+1}=\tau_k(z_k-\alpha_{k+1}\nabla f(x_{k+1}))+(1-\tau_k)y_k=x_{k+1}-\frac{1}{L_{k+1}}\nabla f(x_{k+1}).\] This means that the two methods are not just very similar, but are practically identical. The only difference between them is the usage of exact line search instead of a fixed-length gradient descent step.

\subsection{Convergence Analysis}
To ascertain the convergence of the above algorithm we will require the following lemmas:

\begin{lemma}
For any $u\in\mathbb{R}^n$

\[\alpha_{k+1}\langle\nabla f(x_{k+1}),z_k-u\rangle\leqslant\alpha_{k+1}^2L_{k+1}\left(f(x_{k+1})-f(y_{k+1})+\frac{\tau_k\eps}{2}\right)+\frac{1}{2}\|z_k-u\|^2-\frac{1}{2}\|z_{k+1}-u\|^2. \]
\end{lemma}

\begin{proof}
\begin{align*}
\alpha_{k+1}&\langle\nabla f(x_{k+1}),z_k-u\rangle = \alpha_{k+1}\langle\nabla f(x_{k+1}),z_k-z_{k+1}\rangle+\alpha_{k+1}\langle\nabla f(x_{k+1}),z_{k+1}-u\rangle \\
&\stackrel{\scriptsize{\circled{1}}}{=} \alpha_{k+1}\langle\nabla f(x_{k+1}),z_k-z_{k+1}\rangle+\langle z_k-z_{k+1} ,z_{k+1}-u\rangle\\
&\stackrel{\scriptsize{\circled{2}}}{=}\alpha_{k+1}\langle\nabla f(x_{k+1}),z_k-z_{k+1}\rangle+\frac{1}{2}\|z_k-u\|^2-\frac{1}{2}\|z_{k+1}-u\|^2-\frac{1}{2}\|z_k-z_{k+1}\|^2\\
&\stackrel{\scriptsize{\circled{3}}}{\leqslant} \alpha_{k+1}^2L_{k+1}\left(f(x_{k+1})-f(y_{k+1})+\frac{\tau_k\eps}{2}\right)+\frac{1}{2}\|z_k-u\|^2-\frac{1}{2}\|z_{k+1}-u\|^2.
\end{align*}

Here, $\circled{1}$ is due to \[z_{k+1}=\argmin\limits_{z\in\mathbb{R}^n} \langle\alpha_{k+1}\nabla f(x_{k+1}),z\rangle+\frac{1}{2}\|z_k-z\|^2,\] which implies \[\nabla \left(\frac{1}{2}\|z_k-z\|^2+\langle\alpha_{k+1}\nabla f(x_{k+1}),z\rangle\right)\bigg\rvert_{z=z_{k+1}}=0.\] $\circled{2}$ follows from the triangle equality of Bregman divergence 
\[\langle -\nabla V_x(y),y-u\rangle = V_x(u)-V_y(u)-V_x(y),\] which takes the following form when $V_x(y)=\frac{1}{2}\|x-y\|^2$: \[\langle x-y,y-u\rangle=\frac{1}{2}\|x-u\|^2-\frac{1}{2}\|y-u\|^2-\frac{1}{2}\|x-y\|^2\]
Finally, $\circled{3}$ is due to our choice of $L_{k+1}.$

\end{proof}

\begin{lemma}
For any $u\in\mathbb{R}^n$

\begin{align*}
\alpha_{k+1}^2L_{k+1}f(y_{k+1})-&\left(\alpha^2_{k+1}L_{k+1}-\alpha_{k+1}\right)f(y_k)+\\&
\left(\frac{1}{2}\|z_{k+1}-u\|^2-\frac{1}{2}\|z_k-u\|^2\right)-\frac{\alpha_{k+1}\eps}{2}\leqslant\alpha_{k+1}f(u).
\end{align*}
\end{lemma}

\begin{proof}
We deduce the following sequence of relations:
\begin{align*}
&\alpha_{k+1}(f(x_{k+1})-f(u))\leqslant \alpha_{k+1}\langle \nabla f(x_{k+1}), x_{k+1}-u\rangle\\&=\alpha_{k+1}\langle \nabla f(x_{k+1}), x_{k+1}-z_k\rangle+\alpha_{k+1}\langle \nabla f(x_{k+1}), z_k-u\rangle\\
&\stackrel{\scriptsize{\circled{1}}}{=} \frac{(1-\tau_k)\alpha_{k+1}}{\tau_k}\langle \nabla f(x_{k+1}), y_k-x_{k+1}\rangle+\alpha_{k+1}\langle \nabla f(x_{k+1}), z_k-u\rangle\\
&\stackrel{\scriptsize{\circled{2}}}{\leqslant} \frac{(1-\tau_k)\alpha_{k+1}}{\tau_k}(f(y_k)-f(x_{k+1}))+\alpha^2_{k+1}L_{k+1}\left(f(x_{k+1})-f(y_{k+1})+\frac{\tau_k\eps}{2}\right)\\&+\frac{1}{2}\|z_{k}-u\|^2-\frac{1}{2}\|z_{k+1}-u\|^2
\stackrel{\scriptsize{\circled{3}}}{=} (\alpha^2_{k+1}L_{k+1}-\alpha_{k+1})f(y_k)-\alpha_{k+1}^2L_{k+1}f(y_{k+1})\\&+\alpha_{k+1}f(x_{k+1})+\left(\frac{1}{2}\|z_{k}-u\|^2-\frac{1}{2}\|z_{k+1}-u\|^2\right)+\frac{\alpha_{k+1}\eps}{2}.
\end{align*} Here, $\circled{1}$ uses the fact that our choice of $x_{k+1}$ satisfies $\tau_k(x_{k+1}-z_k)=(1-\tau_k)(y_k-x_{k+1})$. $\circled{2}$ is by convexity of $f(\cdot)$ and \textsc{Lemma 2.2}, while $\circled{3}$ uses the choice of $\tau_k=\frac{1}{\alpha_{k+1}L_{k+1}}$.
\end{proof}

We are now ready to begin our proof of the method's convergence.

\begin{theorem}
Let $f(x)$ be a convex, differentiable function such that its gradient satisfies the H$\ddot{\text{o}}$lder condition for some $\nu\in[0,1]$ with some finite $M_{\nu}$. Let $L_0$ also satisfy 
\[L_0\leqslant \inf_{\nu\in[0,1]}4\left[\frac{1-\nu}{1+\nu}\frac{M_\nu}{\eps}\right]^{\frac{1-\nu}{1+\nu}}M_\nu.\] Then ULCM($f$, $L_0$, $x_0$, $\eps$, $T$) outputs $y_T$ such that $f(y_T)-f(x^\ast)\leqslant\eps$ in the number of iterations

\[T\leqslant\inf_{\nu\in[0,1]}\ \left[\frac{1-\nu}{1+\nu}\right]^\frac{1-\nu}{1+3\nu}\left[\frac{2^\frac{3+5\nu}{2}M_\nu}{\eps}\right]^\frac{2}{1+3\nu}\Theta^\frac{1+\nu}{1+3\nu},\]where $\Theta$ is any upper bound on $\frac{1}{2}\|x_0-x^\ast\|^2$.
\end{theorem}
\begin{proof}
Note that our choice of $\alpha_{k+1}$ satisfies
\[\alpha^2_{k+1}L_{k+1}-\alpha_{k+1}=\alpha^2_{k}L_k,\addtag\] which allows us to telescope \textsc{Lemma 2.3}. Summing up \textsc{Lemma 2.3} for $k=0,1,\ldots, T-1$ and $u=x^\ast$, we obtain

\[\alpha_{T}^2L_{T}f(y_T)+\left(\frac{1}{2}\|z_T-x^\ast\|^2-\frac{1}{2}\|z_0-x^\ast\|^2\right)\leqslant\sum_{k=1}^T\alpha_kf(x^\*) +\sum_{k=1}^T\frac{\alpha_k\eps}{2}.\] By using $(1)$ we get that $\sum\limits_{k=1}^T\alpha_k=\alpha^2_TL_T$. We also notice that $\frac{1}{2}\|z_t-x^\ast\|^2\geqslant 0$ and $\frac{1}{2}\|z_0-x^\ast\|^2\leqslant \Theta$. Therefore, 

\[f(y_T)-f(x^\ast)\leqslant \frac{\Theta}{\alpha^2_TL_T}+\frac{\eps}{2}.\]

Note that our process of calculating $L_k$ guarantees that if the step $L_{k+1}\gets 2L_{k+1}$ of the algorithm was executed at least once for some $k$, then for that $k$

\[L_{k+1}\leqslant 2\left[\frac{1-\nu}{1+\nu}\frac{M_\nu}{\eps\tau_k}\right]^{\frac{1-\nu}{1+\nu}}M_\nu. \addtag\]

Assume that $L_n\leq 2\left[\frac{1-\nu}{1+\nu}\frac{M_\nu}{\eps\tau_{n-1}}\right]^{\frac{1-\nu}{1+\nu}}M_\nu$ and $L_{n+1}=\frac{L_n}{2}$ for some $n\geq 1$. Then \[\frac{1}{\tau_n}=\alpha_{n+1}L_{n+1}=\frac{1}{2}+\sqrt{\frac{1}{4}+\alpha_n^2L_nL_{n+1}}\geq\frac{1}{2}+\sqrt{\frac{1}{4}+\frac{\alpha_n^2L_n^2}{2}}\geq\frac{1}{\sqrt{2}\tau_{n-1}}.\]

\[L_{n+1}=\frac{L_n}{2}\leq \left[\frac{1-\nu}{1+\nu}\frac{\sqrt{2}M_\nu}{\eps\tau_{n}}\right]^{\frac{1-\nu}{1+\nu}}M_\nu\leq2\left[\frac{1-\nu}{1+\nu}\frac{M_\nu}{\eps\tau_{n}}\right]^{\frac{1-\nu}{1+\nu}}M_\nu.\] This shows that even if we don't execute the step $L_{k+1}\gets 2L_{k+1}$, (2) remains true as long as it held true on the previous iteration. All of the above proves that the assumption about $L_0$ in the statement of the theorem implies that (2) is true for all $k=0,\ldots T-1$. 

Denote $A_k=\alpha^2_kL_k$. We may now proceed to attain a lower bound on $A_T$.

\[\frac{\alpha^2_k}{A_k}=\frac{1}{L_k}\geqslant\frac{1}{2M_\nu}\left[\frac{1+\nu}{1-\nu}\frac{\eps}{M_\nu}\right]^{\frac{1-\nu}{1+\nu}}\left[\frac{\alpha_k}{A_k}\right]^{\frac{1-\nu}{1+\nu}}\]

\[\alpha_k\geqslant\frac{1}{2^\frac{1+\nu}{1+3\nu}M_\nu^\frac{2}{1+3\nu}}\left[\frac{1+\nu}{1-\nu}\eps\right]^\frac{1-\nu}{1+3\nu}A_k^\frac{2\nu}{1+3\nu}.\]

Denote $\gamma=\frac{1+\nu}{1+3\nu}\geqslant \frac{1}{2}.$ Since $A_{k+1}=A_k+\alpha_{k+1},$ 

\[A^\gamma_{k+1}-A^\gamma_{k+1}\geqslant \frac{A_{k+1}-A_{k}}{A_{k+1}^{1-\gamma}+A_k^{1-\gamma}}\geqslant\frac{\alpha_{k+1}}{2A_{k+1}^{1-\gamma}}\geqslant\frac{1}{2^\frac{2+4\nu}{1+3\nu}M_\nu^\frac{2}{1+3\nu}}\left[\frac{1+\nu}{1-\nu}\eps\right]^\frac{1-\nu}{1+3\nu}. \addtag\]Now we telescope (3) for $k=0,\ldots,T-1$ and  get

\[A_T\geqslant\left[\frac{1+\nu}{1-\nu}\right]^\frac{1-\nu}{1+\nu}\frac{T^\frac{1+3\nu}{1+\nu}\eps^\frac{1-\nu}{1+\nu}}{2^\frac{2+4\nu}{1+\nu}M_\nu^\frac{2}{1+\nu}}.\] This allows us to estimate the number of iterations necessary to achieve error no more than $\eps$. However, beforehand we shall note that this estimate heavily depends on $\nu$. By allowing $M_\nu$ to be infinite, we make the gradient of any differentiable function satisfy the H$\ddot{\text{o}}$lder condition for all $\nu\in[0,1]$. This in turn allows to easily select the most appropriate estimate:

\[T\leqslant\inf_{\nu\in[0,1]}\ \left[\frac{1-\nu}{1+\nu}\right]^\frac{1-\nu}{1+3\nu}\left[\frac{2^\frac{3+5\nu}{2}M_\nu}{\eps}\right]^\frac{2}{1+3\nu}\Theta^\frac{1+\nu}{1+3\nu}.\]

Note that since the solution $x^\ast$ was arbitrary, $x^\ast$ may now be considered to be the solution which minimizes $\frac{1}{2}\|x_0-x^\ast\|^2$.

\end{proof}
\subsection{Stopping criterion}

In \cite{anikin2017dual} it is shown that the original version of the Linear Coupling Method may be equipped with a stopping criterion. By using similar techniques, we are now going to show that our universal modification of said method may also be equipped with a calculable stopping criterion.

By ignoring the first inequality in the proof of \textsc{Lemma 2.3}, we get that for all $u\in\mathbb{R}^n$ (remember that $A_k=\alpha^2_k L_k$)
\begin{align*}
A_{k+1} f(y_{k+1})-A_k f(y_k)&+\frac{1}{2}\|z_{k+1}-u\|^2-\frac{1}{2}\|z_k-u\|^2-\frac{\alpha_{k+1}\eps}{2}\\&\leq \alpha_{k+1}\left(f(x_{k+1})+\langle\nabla f(x_{k+1}),u-x_{k+1}\rangle\right).
\end{align*}
 Summing up for $k=0,\ldots,m-1$, we obtain 

\[f(y_m)\leq \frac{\eps}{2}+\frac{1}{A_m}\min_{u\in\mathbb{R}^n}\left\lbrace\frac{1}{2}\|z_0-u\|^2+\sum_{i=1}^m\alpha_{i}\left(f(x_{i})+\langle\nabla f(x_{i}),u-x_{i}\rangle\right)\right\rbrace.\]

Denote \[l_m(u)=\sum_{i=1}^m\left[ \alpha_{i}\left(f(x_{i})+\langle\nabla f(x_{i}),u-x_{i}\rangle\right)\right]\] and

\[\hat{f}_m=\min_{u:\ \frac{1}{2}\|z_0-u\|^2\leq\Theta} \frac{1}{A_m} l_m(u).\]

Then by using strong duality one may see that

\begin{align*}
\hat{f}_m&=\min_{u\in \mathbb{R}^n}\max_{\lambda\geq 0}\left\lbrace\frac{1}{A_m} l_m(u)+\lambda\left(\frac{1}{2}\|z_0-u\|^2-\Theta\right)\right\rbrace\\&=\max_{\lambda\geq 0} \min_{u\in \mathbb{R}^n}\left\lbrace\frac{1}{A_m} l_m(u)+\lambda\left(\frac{1}{2}\|z_0-u\|^2-\Theta\right)\right\rbrace.
\end{align*}

By setting $\lambda=\frac{1}{A_m}$, we get that 

\[\hat{f}_m\geq\frac{1}{A_m}\min_{u\in\mathbb{R}^n}\left\lbrace\frac{1}{2}\|z_0-u\|^2+\sum_{i=1}^m\alpha_{i}\left(f(x_{i})+\langle\nabla f(x_{i}),u-x_{i}\rangle\right)\right\rbrace-\frac{\Theta}{A_m}.\]

Then $f(y_m)-\hat{f}_m\leq\frac{\eps}{2}+\frac{\Theta}{A_m}.$ This means that our method is primal-dual. By the convexity of $f$ we also get that $f(x^\ast)\geq\hat{f}_m$, so $f(y_m)-f(x^\ast)\le f(y_m)-\hat{f}_m\leq\eps$ may be used as an implementable stopping criterion. Of course, an estimate of $\Theta$ is required to compute $\hat{f}_m$. Overestimating $\Theta$ may lead to performing an excessive amount of iterations, while underestimating it invalidates the criterion completely. However, the stopping criterion requires an estimate of only one unknown parameter, which is also not used in the algorithm's definition. On the other hand, three unknown parameters ($\nu, M_\nu, \Theta$) need to be estimated to calculate the upper bound on the number of iterations required to get an $\eps$-accurate solution
\section{Line search}

During all of the previous analysis we assumed that $\forall x\in\mathbb{R}^n\ f(x),\ \nabla f(x)$, the steepest descent step, and the mirror descent step may be calculated exactly. However, in relation to the steepest descent step this assumption is not critical for the method's convergence.

For any convex function of one real argument defined on a segment of the form $[a,b]$ of length $l=b-a$ a point $y$ such that \[\|y-\argmin_{x\in[a,b]}f(x)\|\leqslant \eps \] may be found in $O(\log\frac{l}{\eps})$ function value calculations by using the bisection method. However, to perform an exact line search in our algorithm one needs to localize the solution first. To do that we propose the following simple procedure:


\begin{algorithm}
    \SetKwInOut{Input}{Input}
    \SetKwInOut{Output}{Output}
	
    \caption{Localize(f,$l_0$)}
    \Input{$f(x)$ -- convex function defined on $[0,+\infty)$; initial segment length $l_0$.}
    \Output{$l$ such that $\argmin\limits_{x\in[0,+\infty)} f(x) \in [0,l]$}
    $l\gets l_0$\\
    \While{$f(2l)\leqslant f(l)$}{
  		$l\gets2l$
    }
  \Return{$l$}
\end{algorithm}

Let us estimate the accuracy with which the steepest descent must be performed to guarantee our method's convergence. 
Let's say we want to get a solution with accuracy of $\eps+\delta$, where $\delta$ is the term resulting from the inaccuracy of the steepest descent step. To do that we need to slightly modify our algorithm: 
\newpage
\begin{algorithm}
    \SetKwInOut{Input}{Input}
    \SetKwInOut{Output}{Output}
	
    \caption{$\delta$-ULCM($f$, $L_0$, $x_0$, $\eps$, $\delta$, $T$)}
    \Input{$f$ a differentiable convex function with H$\ddot{\text{o}}$lder continuous gradient;
    initial value of the "inexact" Lipschitz continuity constant $L_0$;
    initial point $x_0$;
    accuracy $\eps$;
    line search accuracy $\delta$;
    number of iterations $T$.}
    $y_0 \gets x_0$, $z_0 \gets x_0$, $\alpha_0 \gets 0$\\
  \For{$k=0 \to T-1$}{
  	$L_{k+1}\gets\frac{L_{k}}{2}$\\
    \While{True}{
  	$\alpha_{k+1}\gets\frac{1}{2L_{k+1}}+\sqrt{\frac{1}{4L^2_{k+1}}+\alpha^2_k\frac{L_k}{L_{k+1}}}$\\
    $\tau_k\gets\frac{1}{\alpha_{k+1}L_{k+1}}$\\
    $x_{k+1}\gets\tau_kz_k+(1-\tau_k)y_k$\\
    Choose $y_{k+1}$ such that $f(y_{k+1})\leq \argmin\limits_{h\geq 0} f(x_{k+1}-h\nabla f(x_{k+1}))+\frac{\tau_k\delta}{2}$\\
  $z_{k+1}\gets\argmin\limits_{z\in\mathbb{R}^n}\ \langle\alpha_{k+1}\nabla f(x_{k+1}),z-z_k\rangle +\frac{1}{2}\|z_k-z\|^2$\\
    \If{$\langle \alpha_{k+1}\nabla f(x_{k+1}),z_k-z_{k+1}\rangle-\frac{1}{2}\|z_k-z_{k+1}\|^2\leq \alpha^2_{k+1}L_{k+1}(f(x_{k+1})-f(y_{k+1})+\frac{\tau_k\eps}{2})$}{\textbf{break}}
    \Else{$L_{k+1}\gets 2L_{k+1}$}
  	}
 
  }
  \Return{$y_T$}
\end{algorithm}

\begin{theorem}
Let $f(x)$ be a convex, differentiable function such that its gradient satisfies the H$\ddot{\text{o}}$lder condition for some $\nu\in[0,1]$ with some finite $M_{\nu}$. Let $L_0$ also satisfy 
\[L_0\leqslant \inf_{\nu\in[0,1]}4\left[\frac{1-\nu}{1+\nu}\frac{M_\nu}{\eps}\right]^{\frac{1-\nu}{1+\nu}}M_\nu.\] Then $\delta$-ULCM($f$, $L_0$, $x_0$, $\eps$, $\delta$, $T$) outputs $y_T$ such that $f(y_T)-f(x^\ast)\leqslant\eps+\delta$ in the number of iterations

\[T\leqslant\inf_{\nu\in[0,1]}\ \left[\frac{1-\nu}{1+\nu}\right]^\frac{1-\nu}{1+3\nu}\left[\frac{2^\frac{3+5\nu}{2}M_\nu}{\eps}\right]^\frac{2}{1+3\nu}\Theta^\frac{1+\nu}{1+3\nu},\]where $\Theta$ is any upper bound on $\frac{1}{2}\|x_0-x^\ast\|^2$.
\end{theorem}
 
This immediately follows from the proof of $\textsc{Theorem 2.4}$. To see that, note that if for some $L_{k+1}$ and the exact solution of the line search problem $\hat{y}_{k+1}$

\[\langle \alpha_{k+1}\nabla f(x_{k+1}),z_k-z_{k+1}\rangle-V_{z_k}(z_{k+1})\leq \alpha^2_{k+1}L_{k+1}\left(f(x_{k+1})-f(\hat{y}_{k+1})+\frac{\tau_k\eps}{2}\right)\] holds true, then by definition of $y_{k+1}$ we have
\[\langle \alpha_{k+1}\nabla f(x_{k+1}),z_k-z_{k+1}\rangle-V_{z_k}(z_{k+1})\leq \alpha^2_{k+1}L_{k+1}\left(f(x_{k+1})-f(y_{k+1})+\frac{\tau_k(\eps+\delta)}{2}\right).\] This leads to an analogue of \textsc{Lemma 2.1}. Then by proceeding with the proof the same way it was done in \textsc{Theorem 2.4}, one gets the desired result.

\subsection{Simplified function evaluation during line search}

As noted in \cite{SESOP}, for objectives of particular form the steepest descent step may be performed significantly faster.

Consider a function of the form 

\[f(x)=\phi(\bm{\rm{A}}x)+\psi(x),\] where $x\in \mathbb{R}^n$, $\bm{\rm{A}}\in \mathbb{R}^{n\times n}$.

If $n$ is sufficiently large, the computation of $\bm{\rm{A}}x$ may be the most time-consuming operation during computation of $f(x)$. However, if we are performing the steepest descent step, we can be sure that $x$ is of the form $x_k+\alpha\nabla f(x_k)$. Then

\[\bm{\rm{A}}x=\bm{\rm{A}}x_k+\alpha \bm{\rm{A}}\nabla f(x_k)=v_0+\alpha v_1.\] This shows that one may calculate the two points $v_0$ and $v_1$ just once at the beginning of a steepest descent step. 

If $\psi(y)$ and $\phi(y)$ with $y$ known may be calculated in $\mathcal{O}(n)$ arithmetic operations, then this representation of $\bm{\rm{A}}x$ allows us to evaluate $f(x)$ in $\mathcal{O}(n)$ arithmetic operations after performing matrix multiplication, which requires $\mathcal{O}(n^2)$ arithmetic operations, only twice. This may significantly decrease the cost of one steepest descent step.

\section{Numerical experiments}

The proposed methods were implemented in C$++$ and tested using the
modern versions of GCC, clang and ICC (Intel C Compiler) on both GNU/Linux,
Mac OS X and Microsoft Windows operating systems. The source code is available at \url{http://github.com/htower/ulcm}.

For the presented computational experiments we have also implemented a variant of the conjugate gradients method proposed by Y.~Nesterov in \cite{nesterov1989book}, which we denote as NCG. The method has high numerical stability and a number of interesting properties. In particular, it lacks a restart procedure. This results in an increased iteration complexity relatively to "classic" conjugate gradient methods, which may be attributed to the necessity of solving two line search problems at each iteration. Details are presented in Algorithm \ref{alg_ncg} and Figure \ref{fig_ncg}.

\begin{algorithm}
  \SetKwInOut{Input}{Input}
  \SetKwInOut{Output}{Output}

  \caption{NCG($f$, $x_0$, $\delta$, $T$)}
  \Input{$f$ a differentiable convex function with H$\ddot{\text{o}}$lder continuous gradient;
  initial point $x_0$;
  line search accuracy $\delta$;
  number of iterations $T$.}
  $y_{-2} \gets x_0$,
  $y_{-1} \gets x_0$,
  $y_{ 0} \gets x_0$ \\
  \For{$k = 0$ to $T-1$}
  {
    $\alpha_k \gets \argmin\limits_{\alpha \in \mathbb{R}} f(x_k + \alpha (y_{k-2} - x_k))$\\
    $y_k=x_k+\alpha_k(y_{k-2} - x_k)$\\
    $\beta_k \gets \argmin\limits_{\beta\ge 0} f(y_k - \beta \nabla f(y_k))$\\
    $x_{k+1}=y_k-\beta_k\nabla f(y_k)$
  }
  \Return{$x_T$}
  \label{alg_ncg}
\end{algorithm}

\begin{figure}
\includegraphics[]{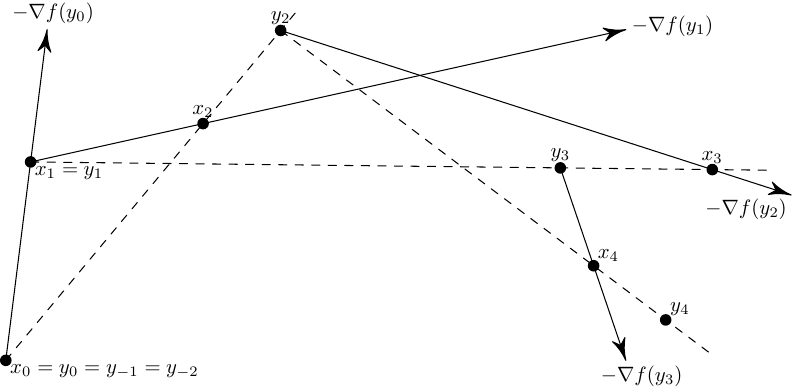}
\caption{Illustration of the NCG method.}
\label{fig_ncg}
\end{figure}

The behaviour of the proposed methods was investigated by a series of numerical experiments on different smooth and non-smooth optimization problems. For all experiments we set the starting point $x_0$ to $10 \cdot e$, where $e = (1,...,1)$, and the precision value $\varepsilon = 10^{-4}$. The methods were interrupted as soon as the objective function's value became lower than $f(x^*) + 5\varepsilon = f(x^*) + 5\times 10^{-4}$. The dimensionality of the problem was up to $10^6$ .

Firstly, we considered the following smooth (quadratic) problem:
\begin{equation}
  f(x) = \sum_{i=1}^{n} i x_i^2 .
\label{eq_problem_s}
\end{equation}
This function is $L$-smooth, but the parameter $L$ depends on the number of dimensions $n$ linearly. This minimization problem can be solved analytically, the optimal value $f(x^\ast)$ is equal to $0$. The results of our experiments are presented in Table~\ref{tbl_s} and Figure~\ref{fig_s36}.
\begin{table}
\begin{tabular}{r|l||r|r|r|r|r|r}
\multirow{2}{26mm}{$n$, problem size} & \multirow{2}{10mm}{$f(x_0)$} & \multicolumn{2}{c|}{UFGM} & \multicolumn{2}{c|}{ULCM} & \multicolumn{2}{c}{NCG} \\
 & & iterations & t, sec. & iterations & t, sec. & iterations & t, sec. \\ \hline
$10^3$ & $5 \cdot 10^7$    &   743 & 0.035 &   722 & 0.035 &  121 & 0.004 \\
$10^4$ & $5 \cdot 10^9$    &  3230 & 1.429 &  3459 & 3.233 &  385 & 0.079 \\
$10^5$ & $5 \cdot 10^{11}$ & 15231 & 141.2 & 18053 & 372.6 & 1217 & 2.796 \\
$10^6$ & $5 \cdot 10^{13}$ & 73185 & 6857 & 84117 & 22373 & 3850 & 98.40 \\
\end{tabular}
\caption{Method's complexity for the smooth problems.}
\label{tbl_s}
\end{table}

\begin{figure}
\includegraphics[]{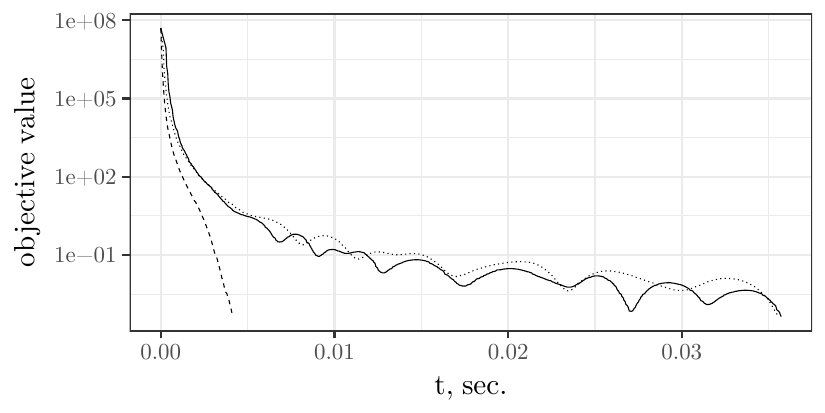}
\includegraphics[]{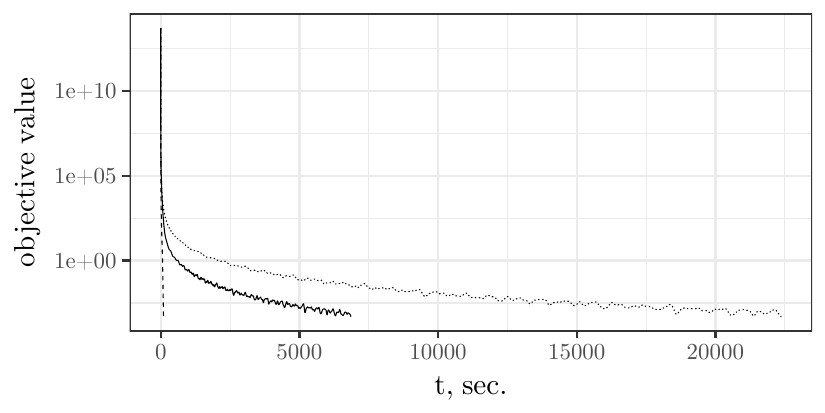}
\caption{Methods convergence for the smooth problems with $n = 10^3$ (top) and $n = 10^6$ (bottom). The solid line stands for the UFGM method, the dotted line stands for the ULCM method, the dashed line stands for the NCG method.}
\label{fig_s36}
\end{figure}

Next, we consider the following non-smooth problem:
\begin{equation}
  f(x) = \max_{i=1,...,n} x_i + \frac{\mu}{2} \| x \|_2^2.
\label{eq_problem_ns}\end{equation}

In our experiments $\mu=0.1$.
Though this function is differentiable almost everywhere. Though it does not have globally H$\ddot{\text{o}}$lder continuous gradients, the gradient satisfies the H$\ddot{\text{o}}$lder continuity condition on any bounded set. 

This minimization problem can be solved analytically, the optimal value $f(x^\ast)$ is equal to $-\frac{1}{2\mu n}=-\frac{5}{n}$.

The gradient (subgradient, in case $f$  is not differentiable at $x$) can be evaluated as
\[
  \nabla f(x) = \mu x + z(x), \quad
  z(x) = (0,...0,1,0,...0) ,
\]
where $1$ is located at position $k = \argmin\limits_{i=1,...,n} x_i .$

The results are shown in Table~\ref{tbl_ns} and Figure~\ref{fig_ns3_ns6}.

\begin{table}
\begin{tabular}{r|l||r|r|r|r}
\multirow{2}{26mm}{$n$, problem size} & \multirow{2}{10mm}{$f(x_0)$} & \multicolumn{2}{c|}{UFGM} & \multicolumn{2}{c}{ULCM} \\
 & & iterations & t, sec. & iterations & t, sec. \\ \hline
$10^3$ & $1 \cdot 10^4$ &  535795 &  17.48 & 1376 & 0.175 \\
$10^4$ & $1 \cdot 10^5$ &  706870 &  233.8 & 6930 & 6.059 \\
$10^5$ & $1 \cdot 10^6$ & 1751285 &   4713 & 6950 & 34.18 \\
$10^6$ & $1 \cdot 10^7$ & 4341186 & 165435 & 6977 & 575.1 \\
\end{tabular}
\caption{Method's complexity for the non-smooth problems.}
\label{tbl_ns}
\end{table}

\begin{figure}
\includegraphics[]{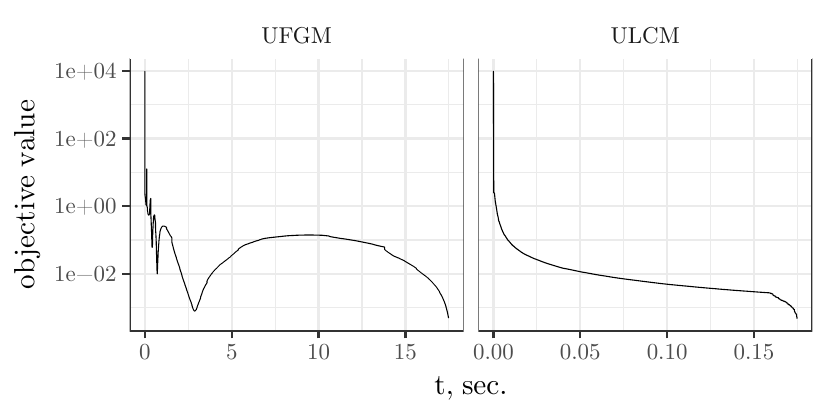}
\includegraphics[]{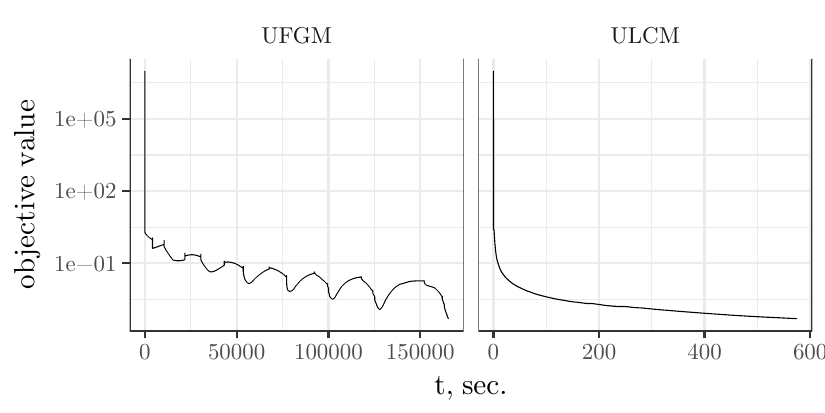}
\caption{Methods convergence for the non-smooth problems with $n = 10^3$ (top) and $n = 10^6$ (bottom).}
\label{fig_ns3_ns6}
\end{figure}

Note that in our particular case, since the ULCM and UFGM methods become identical if the steepest descent of the ULCM methods is replaced with a gradient descent step with step length $\frac{1}{L_{k+1}}$, all the differences in actual performance may be attributed to the line search procedure.

The results of our experiments may be summarized as follows:
\begin{enumerate}
  \item For the smooth problems (\ref{eq_problem_s}) the NCG method showed best performance. Its convergence rate significantly exceeds the convergence rates of UFGM and ULCM methods by up to two orders of magnitude. Although the ULCM method took less iterations to converge, it was slower (about 3 times) in terms of running time. 
  \item For the non-smooth problems (\ref{eq_problem_ns}) the situation is opposite. In that case the ULCM method significantly outperformed UFGM, both in terms of required iterations and elapsed time. In the case of $10^6$ arguments our method converged about 300 times faster.
\end{enumerate}

\section*{Conclusions}
In this paper we propose the first primal-dual method of non-smooth convex optimization with auxiliary line search. Practical experiments show that this method significantly outperforms Nesterov's Universal Fast Gradient Method \cite{nesterov2015universal}. Moreover, we prove that the presented method is also optimal for all the problems with intermediate level of smoothness. The advantage of such an approach is that one can generalize it to stochastic programming using mini-batches \cite{gasnikov2016universal} and to gradient-free methods \cite{dvurechensky2017randomized}.  

\section*{Acknowledgements}

The authors would like to thank Boris Polyak and Yurii Nesterov for helpful comments.

\section*{Funding}
This work was partially supported by RNF 17-11-01027.

\bibliographystyle{plain}
\bibliography{bib.bib}
\end{document}